\documentclass[11pt]{article}
\usepackage{amssymb}
\usepackage{amsmath}
\usepackage{amscd}
\usepackage{amsbsy}
\usepackage{amsfonts}
\usepackage{theorem}
\usepackage{hyperref}

\input xy
\xyoption{all}

\newcommand{\zz}{{\Bbb Z}}

\newcommand{\nn}{{\Bbb N}}

\newcommand{\qq}{{\Bbb Q}}
\newcommand{\pp}{{\Bbb P}}

\newcommand{\ff}{{\Bbb F}}

\newcommand{\ddim}{\operatorname{dim}}
\newcommand{\ddeg}{\operatorname{deg}}

\newcommand{\op}[1]{\operatorname{#1}}
\newcommand{\kbar}{\overline{k}}

\newcommand{\ffi}{\varphi}

\newcommand{\la}{\langle}
\newcommand{\ra}{\rangle}
\newcommand{\row}{\rightarrow}

\newcommand{\low}{\leftarrow}
\newcommand{\lrow}{\longrightarrow}
\renewcommand{\leq}{\leqslant}
\renewcommand{\geq}{\geqslant}

\newcommand{\nichego}[1]{}

\newcommand{\ov}[1]{\overline{#1}}

\newcommand{\wt}[1]{\widetilde{#1}}

\newcommand{\smk}{{\mathbf{Sm_k}}}

\newcommand{\CH}{\op{CH}}

\newcommand{\Qed}{\hfill$\square$\smallskip}
\newcommand{\Red}{\hfill$\triangle$\smallskip}

\newenvironment{proof}{\noindent{\it Proof}:}{\vskip 5mm}

\newtheorem{prop}{Proposition}[section]{\bf}{\it}
\newtheorem{thm}[prop]{Theorem}{\bf}{\it}
{\bf}{\it}
{\bf}{\it}
\newtheorem{defi}[prop]{Definition}{\bf}{\it}
{\bf}{\it}
\newtheorem{observ}[prop]{Observation}{\bf}{\it}
{\bf}{\it}
\newtheorem{rem}[prop]{Remark}{\bf}{}
{\bf}{\it}

{\bf}{\it}
{\bf}{\it}
\newtheorem{cor}[prop]{Corollary}{\bf}{\it}
{\bf}{\it}

\begin{document}

\title{On the $p$-primary and $p$-adic cases of the Isotropy Conjecture}
\author{Alexander Vishik\footnote{School of Mathematical Sciences, University
of Nottingham}}
\date{}
\maketitle

\begin{abstract}
The purpose of this note is to show that, in contrast to the $\ff_p$-case
(proven in \cite{INCHKm}), the $p$-primary and $p$-adic cases of
the Isotropy Conjecture, claiming that the isotropic Chow groups with $\zz/p^r$, $r>1$,
respectively, with $\zz_p$-coefficients over a flexible field coincide with the numerical ones, don't hold. We show that the $BP$-theory with $I(\infty)$-primary, respectively, $I(\infty)$-adic coefficients may serve as a regular substitute for $p$-primary, respectively, $p$-adic Chow groups, which permits to extend the results of \cite{IN} to arbitrary primes.
\end{abstract}

\section{Introduction}

The {\it isotropic realisations} of \cite{Iso} provide local versions of the Voevodsky motivic category. These {\it isotropic motivic categories} 
$\op{DM}(\wt{E}/\wt{E};\ff_p)$ over {\it flexible} fields
are much simpler than the global motivic category and more reminiscent of their
topological counterparts. Such categories possess a natural weight structure in the sense of Bondarko, with the heart - the category of {\it isotropic Chow motives}. The morphisms in the latter category are given by {\it isotropic Chow groups} (with $\zz/p$-coefficients). These are obtained from the usual Chow groups by moding out {\it anisotropic classes}, that is, elements coming as push-forwards from $p$-anisotropic varieties. Any anisotropic class is automatically $\zz/p$-numerically trivial, and we get the natural surjection
$(\CH^*/p)_{iso}\twoheadrightarrow(\CH^*/p)_{Num}$ between the {\it isotropic} and {\it numerical} versions of the theory $\CH^*/p$. The {\it Isotropy Conjecture} claiming that this is an isomorphism over {\it flexible}
fields was proven in \cite[Theorem 4.12]{INCHKm}.
Among other things, this shows that the heart of our weight structure is a semi-simple category with no zero-divisors, which implies that isotropic realisations provide points of the Balmer spectrum of the (compact part of the) Voevodsky motivic category
- see \cite[Theorem 5.13]{INCHKm}.

In \cite[Section 5]{Iso}, the {\it thick} versions of the isotropic motivic categories were introduced. These versions still possess a weight structure with
the heart - the {\it isotropic Chow motives with $\zz/p^r$-coefficients}, for various $r$. The homs there are given by {\it isotropic Chow groups with $\zz/p^r$-coefficients}. In \cite[Theorem 7.2]{IN} it was shown that the {\it Isotropy Conjecture} holds for large primes here, namely, for $p\geq\ddim(X)$,
the $(\CH^*/p^r)_{iso}(X)$ coincides with the numerical version $(\CH^*/p^r)_{Num}(X)$. It was conjectured in \cite[Conjecture 7.1]{IN} that the same
holds always. The purpose of this article is to disprove this $\CH/p^r$-version
of the Isotropy Conjecture (as well as the $\CH(-,\zz_p)$ one). The argument is based on the interplay between the oriented cohomology theories $\CH^*/p^r$ and $BP^*/I(\infty)^r$. The notion of {\it anisotropy} of varieties for both theories is the same - Proposition \ref{CHpr-BPIr}, but the latter theory has the advantage that the Isotropy Conjecture holds for it by \cite[Theorem 4.8]{INCHKm}. Thus, to disprove the 
mentioned conjecture for $\CH^*/p^r$, it is sufficient to exhibit a numerically trivial $\CH^*/p^r$-class which can't be lifted to a numerically trivial
$BP^*/I(\infty)^r$-class. We provide two examples here. In the first one
- see Theorem \ref{c-e}, the 
variety $X$ is the generic quartic 3-fold containing the $2$-Veronese embedding
$Ver^2(C)$ of a conic without rational points and the cycle is the class of the
respective curve $Ver^2(C)$ in $\CH_1(X)/4$ - it is $\CH/4$-numerically trivial,
but not $\CH/4$-anisotropic over any purely transcendental extension, since it
can't be lifted to a numerically trivial $BP^*/I(\infty)^2$-class. Such an example
exist over some purely transcendental extension of any given field.
The second example is provided by a norm-variety $X$ of dimension $p^2-1$ corresponding to a pure
symbol $\{a,b,c\}\in K^M_3(k)/p$ of degree three, the cycle here is the
standard torsion class of dimension $(p-1)$ there - see Theorem \ref{nv}. 
In particular, it is 
$\CH/p^r$-numerically trivial, for any $r$ (as the degree pairing is defined integrally), but it is not $\CH/p^2$-anisotropic (even over the flexible closure),
since it can't be lifted to a numerically trivial $BP^*/I(\infty)^2$ class.
Such an example exists over any field where $K^M_3(k)/p$ is non-zero.
Of course, a norm-variety for a symbol of degree $>3$ will work as well.

This reveals that $p$-primary (but not prime) Chow groups $\CH^*/p^r$, for $r>1$, as an oriented cohomology theory, is not regular enough, and that $BP^*/I(\infty)^r$ may
be considered as a regular substitute for it. Similarly, instead of Chow groups
with $p$-adic coefficients one should look at the completion $BP^*_{I(\infty)}$
of the $BP$-theory at the ideal $I(\infty)$. 
This substitution permits to extend the results of \cite{IN} to arbitrary primes, while their prototypes work only for large ones ($p\geq\ddim(X)$) and fail for small ones by Section \ref{sec-two}.

\section{Examples}
 \label{sec-two}

Everywhere below $k$ is a field of characteristic zero.
Let $I(\infty)=(v_0,v_1,v_2,...)\subset BP$ be the standard invariant ideal
(the kernel of the natural map $BP\row\zz/p$). It was shown in 
\cite[Theorem 4.8]{INCHKm} that, over a {\it flexible} field, for any $r\in\nn$, for the theory
$Q^*=BP^*/I(\infty)^r$, the {\it numerical} version coincides with the {\it isotropic} one. In other words, any numerically trivial $Q$-class is $Q$-anisotropic. In particular, this is true for $r=1$, where $Q^*=\CH^*/p$.

But for $r>1$, the theory $\CH^*/p^r$, in contrast to $BP^*/I(\infty)^r$, 
is not obtained from $BP^*$ by moding out an invariant ideal. At the same time,
the notion of anisotropic varieties for these two theories is identical.

\begin{prop}
 \label{CHpr-BPIr}
Let $X$ be smooth projective. Then TFAE:
\begin{itemize}
 \item[$(1)$] $X$ is $\CH/p^r$-anisotropic;
 \item[$(2)$] $X$ is $BP/I(\infty)^r$-anisotropic.
\end{itemize}
\end{prop}

\begin{proof}
$(2\row 1)$ Clear, since $\CH^*/p^r$ is a quotient of $BP^*/I(\infty)^r$.\\
\noindent $(1\row 2)$ This follows from the fact that
the ideal $I(X)=\op{Im}(BP_*(X)\stackrel{\pi_*}{\row}BP)$ is stable under Landweber-Novikov operations and, 
from any given $w\in BP\backslash I(\infty)^r$, using such operations, one may obtain some $z\in\zz_{(p)}$ not divisible by $p^r$.
 \Qed
\end{proof}

Thus, to show that, for $\CH^*/p^r$, the numerical version of the theory is different from the isotropic one, it is enough to construct a $\CH/p^r$-numerically trivial class which can't be lifted to a numerically trivial class in
$BP^*/I(\infty)^r$.

\begin{thm}
\label{c-e}
Let $k$ be a flexible field. Then there is a smooth projective quartic $3$-fold $X/k$ and a class $u\in\CH_1(X)/4$, such that $u$ is numerically trivial, but not anisotropic in $\CH_1(X)/4$.
\end{thm}

\begin{proof}
 Let $C\subset\pp^2$ be a conic without rational points (always exists over $k$-flexible) and $C'=Ver^2(C)\subset\pp^5$ be its Veronese embedding of $\ddeg=2$.
It is a rational curve of degree $4$ without rational points. 
It is a hyperplane section of $Ver^2(\pp^2)\subset\pp^5$ by a hyperplane $H\cong\pp^4$ (the linear span of $C'$). Note that $C'$ is defined by quadrics in $\pp^5$, since $Ver^2(\pp^2)$ is. Hence, it is also defined by quadrics in $H=\pp^4$. Then it is also defined by hypersurfaces of degree $d$, for any given $d\geq 2$. In particular, it is 
a base set of a linear system of quartics. Among these quartics there is a smooth one. For example, if $(x_0,x_1,x_2)$ are orthogonal coordinates for our conic and 
$q_{i,j}=z_{i,i}^2z_{j,j}^2-z_{i,j}^4$ are "squares" of the standard quadratic relations defining $Ver^2(\pp^2)$ (where $z_{i,j}$ is the Veronese coordinate corresponding to $x_ix_j$),
then the generic linear combination of $q_{0,1}$, $q_{1,2}$ and $q_{0,2}$ restricted to $H$ will be smooth. Hence the generic representative of our linear
system of quartics is smooth as well.

Let our linear system of quartics be parametrised by $\pp^M$ and $Q_{\eta}$ be our generic representative of it. Let $Y=\{(x,Q)\,|\,x\in Q\}\subset(\pp^4\backslash C')\times\pp^M$. We have natural maps:
$(\pp^4\backslash C')\stackrel{\pi}{\low} Y\stackrel{\ffi}{\row}\pp^M$, where
$\pi$ is a $\pp^{M-1}$-bundle. Let $Y_{\eta}$ be the generic fiber of $\ffi$.
That is, $Y_{\eta}=Q_{\eta}\backslash C'$. We have natural surjections:
$$
\CH^*(Y_{\eta})\twoheadleftarrow\CH^*(Y)\twoheadleftarrow\CH^*(\pp^4)[\rho],
$$
where $\rho=c_1(O(1))$, for the canonical line bundle $O(1)$ on our $\pp^{M-1}$-bundle. Since the latter bundle is the restriction of $O(1)$ on $\pp^M$ (via $\ffi$), $\rho|_{Y_{\eta}}=0$ and we get the surjection:
$\CH^*(\pp^4)\twoheadrightarrow\CH^*(Y_{\eta})$.

Applying it to $\CH_1$, we get the surjection $f:\CH_1(\pp^4)\twoheadrightarrow\CH_0(Y_{\eta})$. Note that $Y_{\eta}=Q_{\eta}\backslash C'$, where $C'$ is $2$-anisotropic. So, the degree (mod 2) map
$\ov{deg}: \CH_0(Y_{\eta})\row\zz/2$ is well defined. Since $\ov{deg}(f(l_1))=\ov{deg}(l_1\cdot Q)$, where $Q$ is any representative of our linear system which is a quartic, we get that $Y_{\eta}$ is $2$-anisotropic. Then
so is $Q_{\eta}=Y_{\eta}\cup C'$. 

Applying it to $\CH^1$, we get the surjection
$g:Pic(\pp^4)\twoheadrightarrow Pic(Y_{\eta})=Pic(Q_{\eta})$, so line bundles on $Q_{\eta}$ are restrictions of bundles from the ambient projective space. 

Let $X=Q_{\eta}$ and $u=[C']\in\CH_1(Q_{\eta})/4$. Since $\CH^1(Q_{\eta})/4$ is additively generated by divisors from $\pp^4$ and $C'$ has degree $4$, we get
that $u$ is $\CH/4$-numerically trivial. On the other hand, the element
$u_{BP}=[C']_{BP}\in (BP^*/I(\infty)^2)(Q_{\eta})$ is a lifting of $u$, which
is not numerically trivial (since the push-forward of it to the point is $[\pp^1]=v_1\not\in I(\infty)^2$). If $u$ were {\it anisotropic}, then there would exist another lifting $u'_{BP}$ which would have been 
$BP/I(\infty)^2$-anisotropic. But then $u'_{BP}-u_{BP}=w\cdot v_1$, where
$w$ is the class of a $0$-cycle on $X$. But we saw above that $X$ is $2$-anisotropic. Hence, $w\cdot v_1$ is $BP/I(\infty)^2$-numerically trivial.
This gives a contradiction, since $u'_{BP}$ was and $u_{BP}$ was not $BP/I(\infty)^2$-numerically trivial.

Above, the example of $X$ and $u$ was constructed over the purely transcendental extension $k(\pp^M)$ of the base field, but since $k$ is flexible,
an example with the identical properties exists over $k$ itself.
 \Qed
\end{proof}

This shows that, in contrast to the prime coefficients case proven in \cite[Theorem 4.12]{INCHKm}, the $p$-primary case of the \cite[Conjecture 7.1]{IN} doesn't hold, in general. Note, however, that due to \cite[Theorem 7.2]{IN}, this conjecture holds for sufficiently large primes, namely, for $\ddim(X)\leq p$. In particular, the counter-example of Theorem \ref{c-e} is the smallest possible (in terms of dimension).

This demonstrates that the Chow groups $\CH^*/p^r$ with $p$-primary (as opposed to {\it prime}) coefficients is not as regular a cohomology theory, as one would wish, and the right "regular" substitute for it is
$BP^*/I(\infty)^r$, for which the Isotropy Conjecture holds (by \cite[Theorem 4.8]{INCHKm}). 

\begin{thm}
 \label{nv}
Let $k$ be a field of characteristic zero with $K^M_3(k)/p\neq 0$. Then there exists a variety $X$ of dimension $p^2-1$ over $k$ with a $p$-torsion (and so, $\CH$-numerically trivial) class $u\in\CH^*(X)$ which is not
$\CH^*/p^2$-anisotropic.
\end{thm}

\begin{proof}
As a variety $X$ we will use the norm-variety corresponding to a
non-zero pure symbol $\alpha=\{a,b,c\}\in K^M_3(k)/p$ (which exists by our assumptions).

For $p=2$, as a norm-variety, we may use any $3$-dimensional subquadric $X=P_{\alpha}$ of the $3$-fold Pfister quadric $Q_{\alpha}$.
It was shown by Rost \cite{R} that the motive $M(P_{\alpha})$ of our quadric contains the direct summand $M_{\alpha}$ (we now call the {\it Rost motive}) which over the algebraic closure splits as a sum of two Tate-motives $T\oplus T(3)[6]$. In \cite{R}, Rost computed the Chow groups of $M_{\alpha}$: $\CH_*(M_{\alpha})=\zz\cdot e^0\oplus\zz/2\cdot e_1\oplus\zz\cdot e_0$, where $e^0,e_1,e_0$ are elements of dimension
$3,1$ and $0$, respectively. The class $u=e_1$ is what we need.
 
Due to \cite[Corollary 2.8]{VY}, $M_{\alpha}$ can be lifted uniquely to
a {\it cobordism-motive} $M^{\Omega}_{\alpha}$. The Algebraic Cobordism $\Omega^*$ of Levine-Morel of it was computed in \cite[Theorem 3.5, Proposition 4.4]{VY}. It is shown that the passage $ac$ to the
algebraic closure is injective on $\Omega^*(M^{\Omega}_{\alpha})$,
the classes $e_1$ and $e_0$ may be lifted to $e_1^{\Omega},e_0^{\Omega}\in\Omega_*(M^{\Omega}_{\alpha})$
and $ac(e_1^{\Omega})=v_1\cdot l_0$, $ac(e_0^{\Omega})=2\cdot l_0$,
where $l_0\in\Omega_0(M^{\Omega}_{\alpha}|_{\kbar})=\CH_0(M_{\alpha}|_{\kbar})$
is the class of a rational point. Since $e_1$ is $2$-torsion, it is numerically trivial in
$\CH^*(X)$ and $\CH^*(X,\zz_2)$, since $\zz$ has no torsion. 
On the other hand, it can be lifted to $e_1^{BP}\in BP_1(X)$ which
is not numerically trivial in $BP^*/I(\infty)^2$, since the push-forward of it to the point is $v_1$. Any other choice of lifting will differ by $w\cdot v_1$, where $w\in BP_0(X)$. Since $X$ is $2$-anisotropic, this difference
is numerically trivial in $BP^*/I(\infty)^2$. Thus, there is no lifting of $e_1$ which would be $BP/I(\infty)^2$-numerically trivial. Hence,
$e_1$ is not $2^2$-anisotropic.

For odd $p$, assume first that $k$ contains $p$-th roots of $1$.
We have the norm-varieties $P_{\alpha}$ of Rost \cite{RoNVAC}.
In the case of a $3$-symbol $\{a,b,c\}\in K^M_3(k)/p$, one may use, for example, a compactification of the normed hypersurface
$Nrd(x)=c$ considered by Merkurjev-Suslin, where $Nrd: Quat(\{a,b\})\row k$ is the reduced norm on the quaternion algebra. 
The motive $M(P_{\alpha})$ contains the direct summand $M_{\alpha}$ -
the {\it generalised Rost motive}. The $BP^*$-theory (of the $BP$-version) of it was computed by Yagita \cite[Theorem 10.6]{Y-BP}. In particular, in dimensions $<p$,
it is given by $BP\cdot e^{BP}_{p-1}+BP\cdot e^{BP}_0$ subject to the relation
$p\cdot e^{BP}_{p-1}=v_1\cdot e^{BP}_{0}$, where $e^{BP}_{p-1}$
and $e^{BP}_0$ are classes of dimension $(p-1)$ and zero, respectively,
whose push-forwards to the point are $v_1$, respectively, $v_0=p$.
In particular, the restriction $u=e_{p-1}\in CH_{p-1}(M_{\alpha})$ is a $p$-torsion (and so, a $\CH$-numerically trivial) element, whose $BP^*$-lifting $e_{p-1}^{BP}$ is not $BP/I(\infty)^2$-numerically trivial (since the push-forward of it to the point is $v_1\not\in I(\infty)^2$). On the other hand, any other $BP^*$-lifiting will differ by $v_1\cdot w$, where $w$ is a zero-cycle. Since the variety is $p$-anisotropic, this class is 
$BP/I(\infty)^2$-numerically trivial. Hence, there are no $BP/I(\infty)^2$-numerically trivial liftings, and so, $u$ is not $p^2$-anisotropic.

Above we assumed that $k$ contains $p$-roots of $1$. In general, we may consider $l=k(\sqrt[p]{1})$. If $(X,u)$ is the variety and the class over $l$, then the composition $Y=X\row\op{Spec}(l)\stackrel{\ffi}{\row}\op{Spec}(k)$ and $\ffi_*(u)\in\CH_*(Y)$ will give us what we need over $k$, since $[l:k]$ is prime to $p$.
 \Qed
\end{proof}

Since the above varieties exist over any flexible field, this shows that the
$p$-adic version of the Isotropy Conjecture \cite[Conjecture 7.4]{IN} doesn't hold for small primes. At the same time, it holds for large ones, namely, for $p\geq\ddim(X)$, by \cite[Theorem 7.5]{IN}. 

\section{A regular substitute for $p$-adic Chow groups}

As was pointed above, the $I(\infty)$-primary $BP$-theory has identical isotropy properties and may be considered as a regular substitute for $p$-primary Chow groups. Similarly, the $I(\infty)$-adic $BP$-theory may serve as a regular substitute for Chow groups with $p$-adic coefficients.

Let $BP_{I(\infty)}^*(X)$ be the completion of $BP^*(X)$ at the ideal $I(\infty)$, that is, the limit $\displaystyle\operatornamewithlimits{lim}_r (BP/I(\infty)^r)^*(X)$. For $X$ smooth projective, there is the degree pairing:
$$
BP_{I(\infty)}^*(X)\times BP_{I(\infty)}^*(X)\lrow BP_{I(\infty)}.
$$
Denote as $(BP_{I(\infty)})^*_{Num}(X)$ the numerical version of our $I(\infty)$-adic theory obtained by moding out the kernel $\hat{J}$ of the pairing. Using \cite[Definition 4.3]{Iso} it may be extended to an oriented cohomology theory on $\smk$. 

We have a similar degree pairing for $BP^*(X)$ which commutes with the topological realisation. Thus, 
$BP^*_{Num}(X)=BP^*(X)/J$ is a sub-quotient of 
$\displaystyle BP^*_{Top}(X)$, corresponding to a submodule generated in non-negative co-dimension (since algebraic cobordism of Levine-Morel has such generators - \cite{LM}). Hence, it is a finitely generated $BP$-module. Let $x_1,\ldots,x_m$ be its generators. Then, modulo numerically trivial classes, any element
of $BP^*(X)$ can be written as a $BP$-linear combination of $x_1,\ldots,x_m$. Similarly, modulo $I(\infty)^r\cdot J$, any element
of $I(\infty)^r\cdot BP^*(X)$ can be written as a $I(\infty)^r$-linear
combination of $x_1,\ldots,x_m$. 
Numerically trivial classes
in $BP^*(X)$ will still be numerically trivial in $BP_{I(\infty)}^*(X)$ and so, will belong to $\hat{J}$. Moreover, any sum
$\sum_r j_r$, where $j_r\in I(\infty)^r\cdot J$, belongs to $\hat{J}$,
since our pairing is continuous and so, $\hat{J}$ is closed in the $I(\infty)$-adic topology.
Thus, modulo numerically trivial classes, any element of $BP^*_{I(\infty)}(X)$ may be written as a $BP_{I(\infty)}$-linear combination of $x_1,\ldots,x_m$. Thus, the natural map of $BP_{I(\infty)}$-modules $BP^*_{Num}(X)\otimes_{BP}BP_{I(\infty)}\twoheadrightarrow (BP_{I(\infty)})^*_{Num}(X)$ is surjective and the target module is finitely generated as well.

Since $BP^*(X)\otimes_{\zz_{(p)}}\qq$ is a re-orientation of $\CH^*_{\qq}(X)\otimes_{\qq}BP_{\qq}$, the set of numerically trivial classes doesn't depend on the choice of orientation and $(\CH_{\qq})^*_{Num}(X)$ is a finite-dimensional
$\qq$-vector space, we get:

\begin{prop}
 $BP_{Num}^*(X)\otimes_{\zz_{(p)}}\qq$ is 
a free $BP_{\qq}$-module of finite rank.
\end{prop}

\begin{rem}
Note that $BP^*_{Num}(X)$ itself is not a free $BP$-module, in general.
Consider a $3$-dimensional quadric $P_{\alpha}$ which is a neighbour of a Pfister quadric $Q_{\alpha}$, where $\alpha=\{a,b,c\}\in K^M_3(k)/2$ is a non-zero pure
symbol. Then the motive of $P_{\alpha}$ splits  as a sum $M_{\alpha}\oplus M_{\beta}(1)[2]$ of two Rost motives, where $\beta$ is a divisor of $\alpha$ of degree $2$. By \cite[Corollary 2.8]{VY}, the same happens to the $\Omega$-motive,
and so, to the $BP$-motive. By \cite[Theorem 3.5, Proposition 4.4]{VY}, 
the restriction $ac$ to algebraic closure is injective on $BP$ and it
identifies $BP^*(P_{\alpha})$ with $BP\cdot 1\oplus BP\cdot h\oplus I(1)\cdot l_1\oplus I(2)\cdot l_0$, where  
$1,h,l_1,l_0$ is the standard basis for a split quadric, and $I(1)=(p)$,
$I(2)=(p,v_1)$ are the standard invariant ideals of Landweber. Thus,
there are no (non-zero) numerically trivial elements and so,
$BP^*_{Num}(P_{\alpha})=BP^*(P_{\alpha})$. Here $I(1)$ is a free 
$BP$-module, but $I(2)$ is not. Of course, tensor $\qq$, it will become free.
 \Red 
\end{rem}

By \cite[Proposition 4.1]{INCHKm} we have the action of the algebra
$BP_*BP$ of $BP$-Landweber-Novikov operations on $BP^*_{Num}(X)$. By the result of Landweber \cite[Lemma 3.3]{La73b}, this module is an extension of finitely many modules of the shape
$BP/I(n)$, for $0\leq n\leq\infty$, where $I(n)=(v_0,v_1,\ldots,v_{n-1})$ is the standard invariant ideal of Landweber. 
Since the sequence $v_0,v_1,\ldots,v_{n-1}$ is still regular in 
$BP_{I(\infty)}$, using the Koszul complex, we see that
$\op{Tor}^{BP}_i(BP_{I(\infty)},BP/I(n))=0$, for any $i>0$.
Hence, $\op{Tor}^{BP}_i(BP_{I(\infty)},BP^*_{Num}(X))=0$, for any $i>0$. Similarly, we have the $BP_*BP$-action on
$BP^*_{Num}(X)/p$ and, in particular, 
$\op{Tor}^{BP}_1(BP_{I(\infty)},BP^*_{Num}(X)/p)=0$.
Since $BP^*_{Num}(X)$ has no torsion (as $BP$ has none), the latter fact implies that $BP^*_{Num}(X)\otimes_{BP}BP_{I(\infty)}$
has no torsion either.

There is the natural map $BP^*(X)\otimes_{BP} BP_{I(\infty)}\row BP^*_{I(\infty)}(X)$ commuting with the pairing, which gives the map
$BP^*_{Num}(X)\otimes_{BP}BP_{I(\infty)}\row (BP_{I(\infty)})^*_{Num}(X)$ considered above. 
Since the source object has no torsion and the map tensor $\qq$ is injective (as $BP^*_{Num}(X)\otimes_{\zz_{(p)}}\qq$ is a free $BP_{\qq}$-module), our map is injective. But above we saw that it is also surjective. Thus, we get:
\begin{prop}
\begin{equation*}
(BP_{I(\infty)})^*_{Num}(X)=BP^*_{Num}(X)\otimes_{BP}BP_{I(\infty)}.
\end{equation*}
In particular, 
$(BP_{I(\infty)})^*_{Num}(X)\otimes_{\zz_p}\qq_p$
is a free $(BP_{I(\infty)})_{\qq_p}$-module of finite rank. 
\end{prop}
It also shows that the
numerical version for the theory $BP^*_{I(\infty)}$ coincides with that for the
{\it free} theory $BP^*\otimes_{BP}BP_{I(\infty)}$. Observe, that the theory
$BP^*_{I(\infty)}$ itself is not free and not even {\it constant}.

The natural map
$\displaystyle 
(BP_{I(\infty)})^*_{Num}(X)\row\operatornamewithlimits{lim}_r (BP/I(\infty)^r)_{Num}^*(X)$ is clearly injective.
The images of generators $x_1,\ldots,x_m$ of the $BP$-module
$BP^*_{Num}(X)$ generate $(BP/I(\infty)^r)^*_{Num}(X)$ as a
$(BP/I(\infty)^r)$-module. Observe that the module $BP^*_{Num}(X)$ is completely determined by the pairing matrix of the elements
$x_1,\ldots,x_m$. This matrix involves only finitely many variables $v_i$, and so, our $BP$-module is finitely presented (since a polynomial ring in finitely many variables is Noetherian). 
Denote: $N:=BP^*_{Num}(X)$. Let $(BP/I(\infty)^r)^*_{Num}(X)=N/L(r)$,
where $L(r)\subset N$ be the submodule of
$I(\infty)^r$-numerically trivial classes. 
We know that $N_{\qq}$ is a free $BP_{\qq}$-module. After reordering of variables, we may assume that it is freely generated
by $x_1,\ldots,x_l$. Then $p^t\cdot N\subset \la x_1,\ldots,x_l\ra$, for some $t$. Since $x_1,\ldots,x_l$ are $BP$-numerically independent, there is some
other (possibly, different) $l$-tuple of generators $x_{i_1},\ldots,x_{i_l}$, so that the respective $l\times l$
minor of the pairing matrix is non-zero. Let $s$ be the {\it lower degree} of this minor (i.e., the largest power of $I(\infty)$ containing it). 
Let $y\in L(r)$. Then $p^t\cdot y=\sum_{i=1}^l u_i\cdot x_i$. Since this class is $I(\infty)^{r+t}$-numerically trivial, pairing it with $x_{i_1},\ldots,x_{i_l}$
and using the Kramer's rule and the fact that the graded ring $\op{Gr}_{I(\infty)}(BP)$ is an integral domain, we get that $u_i\in I(\infty)^{r+t-s}$,
for $i=1,\ldots,l$. Thus, $p^t\cdot L(r)\subset I(\infty)^{r+t-s}\cdot N$. 
Denote as $N_d$ the graded part of $N$ of dimension $d$ and similar for submodules. As generators $v_j$ of $BP$, for $j>0$, have dimension $\geq (p-1)$, we obtain that, for $r\geq s+d/(p-1)$, 
$(I(\infty)^{r+t-s}\cdot N)_d=p^t(I(\infty)^{r-s}\cdot N)_d$. Since $N$ has
no torsion, this implies that $L(r)_d\subset (I(\infty)^{r-s}\cdot N)_d$. 
Since $(I(\infty)^r\cdot N)_d\subset L(r)_d\subset (I(\infty)^{r-s}\cdot N)_d$,
we get:
$$
\operatornamewithlimits{lim}_r (BP/I(\infty)^r)_{Num}^*(X)=
\operatornamewithlimits{lim}_r N/I(\infty)^r\cdot N=
\operatornamewithlimits{lim}_r N\otimes_{BP}(BP/I(\infty)^r).
$$
Consider the latter as a functor $F(N)$ of $N$, where 
$F:(BP-mod)\row (BP_{I(\infty)}-mod)$ . We have a natural
map of functors $\ffi: N\otimes_{BP}BP_{I(\infty)}\row F(N)$,
which is an isomorphism for free modules.  
Let $0\row M\row T\row N\row 0$ be a short exact sequence,
where $T$ is a free module. Let $f_r: M\otimes_{BP}B_r\row T\otimes_{BP}B_r$ be the induced map, where $B_r=BP/I(\infty)^r$. Then the short exact sequence
$0\row\op{Im}(f_r)\row T\otimes_{BP}B_r\row N\otimes_{BP}B_r\row 0$ satisfies the Mittag-Leffler condition. So, the map
$\displaystyle\operatornamewithlimits{lim}_r T\otimes_{BP}B_r\row
\operatornamewithlimits{lim}_r N\otimes_{BP}B_r$ is surjective.
Thus, the map $\ffi(N)$ is surjective, in other words, the map
$\displaystyle 
(BP_{I(\infty)})^*_{Num}(X)\row\operatornamewithlimits{lim}_r (BP/I(\infty)^r)_{Num}^*(X)$ is surjective and so, an isomorphism.
Hence, we get:
\begin{prop}
\begin{equation}
 \label{one}
 (BP_{I(\infty)})^*_{Num}(X)=\operatornamewithlimits{lim}_r (BP/I(\infty)^r)_{Num}^*(X).
\end{equation}
\end{prop}
So, the $I(\infty)$-adic numerical $BP$-theory is the limit of the
$I(\infty)$-primary numerical $BP$-theories.
In a similar fashion, we may introduce the {\it isotropic} version of 
$BP_{I(\infty)}^*$.

\begin{defi}
Define the isotropic $I(\infty)$-adic $BP$-theory as:
\begin{equation}
 \label{two}
(BP_{I(\infty)})^*_{iso}(X)=\operatornamewithlimits{lim}_r (BP/I(\infty)^r)_{iso}^*(X).
\end{equation}
\end{defi}

Now we can describe the integral numerically trivial classes
in terms of anisotropic ones. Recall, that a class $u\in BP^*(X)$
is called $I(\infty)^r$-anisotropic, if it comes as a push-forward
from some $BP/I(\infty)^r$-anisotropic (or, which is the same,
a $p^r$-anisotropic) variety. We say that a class is $I(\infty)^{\infty}$-anisotropic, if it is $I(\infty)^r$-anisotropic, for any $r$.

\begin{observ}
Let $k$ be flexible. Then any element from 
$I(\infty)^r\cdot BP^*(X)$ is $I(\infty)^r$-anisotropic.
In particular, $(BP/I(\infty)^r)^*_{iso}(X)=BP^*(X)/(I(\infty)^r-\,\text{anisotropic classes})$.
\end{observ}

\begin{proof}
Any monomial generator $v_{i_1}\cdot\ldots\cdot v_{i_r}$
of $I(\infty)^r$ may be realised as the class of a product 
$V_{i_1}\times\ldots\times V_{i_r}$, where $V_{i_m}$ is a norm-variety \cite{RoNVAC}
for a non-zero pure symbol $\{t_{m,1},\ldots,t_{m,i_m+1}\}\in K^M_{i_m+1}(k)/p$, where $t_{j,l}$ are different transcendental parameters of our flexible field. The fact that this variety is $p^r$-anisotropic is clear by induction on $r$, projecting on factors.
The second claim is now obvious.
 \Qed
\end{proof}

Since every anisotropic class is numerically trivial, we get the natural map of theories
$$
(BP_{I(\infty)})^*_{iso}\row (BP_{I(\infty)})^*_{Num}.
$$

We have the improved analogue of \cite[Theorem 7.5]{IN}:

\begin{thm}
 \label{BPIadic-thm}
Let $k$ be flexible. Then $(BP_{I(\infty)})^*_{iso}=(BP_{I(\infty)})^*_{Num}$. In particular,
$$
BP^*_{Num}(X)=BP^*(X)/(I(\infty)^{\infty}-\text{anisotropic classes}).
$$
\end{thm}

\begin{proof}
The fact that the isotropic version coincides with the numerical 
one follows from
\cite[Theorem 4.8]{INCHKm} together with (\ref{one}) and
(\ref{two}). 

The kernel of the natural map $BP^*(X)\row (BP_{I(\infty)})^*_{Num}(X)$
consists of (integral) numerically trivial classes, while the kernel of the natural map $BP^*(X)\row (BP_{I(\infty)})^*_{iso}(X)$ consists of $I(\infty)^{\infty}$-anisotropic classes. Since these coincide, we get the formula for $BP^*_{Num}(X)$.
 \Qed
\end{proof}

\begin{rem}
The advantage of our Theorem \ref{BPIadic-thm} in comparison to \cite[Theorem 7.5]{IN} is that it holds for any prime $p$, while the latter
result holds only for large primes, namely, for $p\geq\ddim(X)$. Moreover,
for small primes it is just wrong, as was shown in Section \ref{sec-two}.
 \Red
\end{rem}

Since $BP$ has no torsion, we get:

\begin{cor}
Let $k$ be flexible. Then torsion classes in $BP^*(X)$ are $I(\infty)^{\infty}$-anisotropic.
\end{cor}

Let now $k$ be any field (of characteristic zero) and $\wt{k}=k(\pp^{\infty})$
be its {\it flexible closure}. Since {\it free} theories are stable under purely transcendental extensions, we get:

\begin{cor}
Let $k$ be any field and $X/k$ be smooth projective. Then
$$
BP^*_{Num}(X)=BP^*(X_{\wt{k}})/(I(\infty)^{\infty}-\text{anisotropic classes}).
$$
\end{cor}

\end{document}